\documentclass{article}
\usepackage[utf8]{inputenc}
\usepackage[T1]{fontenc} 
\usepackage{amsmath}
\usepackage{amsfonts}
\usepackage{amssymb}
\usepackage{amsthm}
\usepackage{color}
\usepackage{esint}
\usepackage{enumerate}
\usepackage{psfrag}
\usepackage{stmaryrd}
\usepackage{hyperref}
\usepackage{graphicx}
\usepackage{graphics}
\usepackage{epstopdf}
\usepackage{epsfig}
\usepackage{setspace}

\usepackage[sans]{dsfont}
\usepackage{algorithm}
\usepackage{algpseudocode}
\usepackage{subcaption}
\usepackage{pgfplots} % also loads tikz -- generate plots from data/math expressions
\pgfplotsset{
    legend cell align = {left},
    compat=1.9,
    tick pos = left,
    grid = both,
    grid style = {densely dotted, gray}
}
\usetikzlibrary{math}
\usepackage{booktabs} % For \toprule, \midrule and \bottomrule in tables
\usepackage{colortbl} % Allows for shaded background in tables
\usepackage{pgfplotstable} % Generate tables from data

\pgfplotstableread{MsFEM-tests.dat}{\MsFEMtests} % convergence results for various H in H^1 norm
\pgfplotstableread{MsFEM-tests-rel.dat}{\MsFEMtestsRel} % convergence results for various H in H^1 norm normalized by the norm of reference solution
\pgfplotstableread{MsFEM-GalvsPG.dat}{\GalvsPG} % difference between Galerkin and PG solutions in H^1 norm, not normalized
\pgfplotstablecreatecol[copy column from table={\MsFEMtests}{[index] 2}] {LinGal} {\GalvsPG} % add the absolute error committed by Galerkin MsFEM-lin to the data available through \GalvsPG

\newcommand{\bbR}{\mathbb{R}}
\newcommand{\Pone}{
    \ensuremath{ {\mathbb{P}_1} }
}

\newcommand{\phiPone}[1]{
    \ensuremath{ \phi_{#1}^{\mathbb{P}_1} }
}
\newcommand{\phiEps}[1]{
    \ensuremath{ \phi_{#1}^\varepsilon }
}
\newcommand{\VK}[1]{
    \ensuremath{ \chi_K^{\varepsilon,#1} }
}
\newcommand{\dps}{\displaystyle}
\newcommand{\dive}{\operatorname{div}}
\newcommand{\eps}{\varepsilon}

\newtheorem{theorem}{Theorem}
\newtheorem{lemma}[theorem]{Lemma}
\newtheorem{remark}[theorem]{Remark}

\begin{document}

\title{Non-intrusive implementation of Multiscale Finite Element Methods: an illustrative example}

\date{\today}
\author{R. A. Biezemans$^{1,2}$, C. Le Bris$^{1,2}$, F. Legoll$^{1,2}$ and A. Lozinski$^{2,3}$
\\
{\footnotesize $^1$ \'Ecole Nationale des Ponts et Chauss\'ees, 6 et 8 avenue Blaise Pascal,}\\
{\footnotesize 77455 Marne-La-Vall\'ee Cedex 2, France} \\
{\footnotesize $^2$ MATHERIALS project-team, Inria Paris, 2 rue Simone Iff,}\\
{\footnotesize CS 42112, 75589 Paris Cedex 12, France} \\
{\footnotesize $^3$ on partial leave from Laboratoire de Mathématiques de Besançon,}\\
{\footnotesize UMR CNRS 6623, Université de Bourgogne-Franche Comté, France} \\
%{\footnotesize \tt \{rutger.biezemans,claude.le-bris,frederic.legoll\}@enpc.fr, alexei.lozinski@univ-fcomte.fr}\\
}

%% Mathematical classification (2020)
%\subjclass{35J, 35B27, 74Q15, 65N30}
%% For previous classifications, use \subjclass[2010]{00X99}

%65N30   	Finite element, Rayleigh-Ritz and Galerkin methods for boundary value problems involving PDEs

% cocv lemaire et olga num
%35J            elliptic PDEs
%35B27   	Homogenization in context of PDEs; PDEs in media with periodic structure
%74Q15   	Effective constitutive equations in solid mechanics

% alexei chinese
%35B27   	Homogenization in context of PDEs; PDEs in media with periodic structure
%65M60   	Finite element, Rayleigh-Ritz and Galerkin methods for initial value and initial-boundary value problems involving PDEs
%65M12   	Stability and convergence of numerical methods for initial value and initial-boundary value problems involving PDEs

\maketitle

\begin{abstract}
  Multiscale Finite Element Methods (MsFEM) are finite element type approaches dedicated to multiscale problems. They first compute local, oscillatory, problem-dependent basis functions which generate a specific discretization space, and next perform a Galerkin approximation of the problem on that space. We investigate here how these approaches can be implemented in a non-intrusive way, in order to facilitate their dissemination within industrial codes or non academic environments.
\end{abstract}

\section{Introduction}

We consider the highly oscillatory diffusion problem 
\begin{equation}
  -\dive \left( A^\eps \nabla u^\eps \right) = f \ \ \text{in $\Omega$}, \qquad u^\eps = 0 \ \ \text{on $\partial \Omega$},
  \label{eq:diffusion-pde}
\end{equation}
in a bounded domain $\Omega \subset \bbR^d$, where the diffusion coefficient $A^\eps$ is assumed to oscillate on a typical length scale of size $\eps$ much smaller than the diameter of $\Omega$, and to satisfy the usual ellipticity assumptions (see Section~\ref{sec:motivation} below for details).

We seek a numerical approximation of~\eqref{eq:diffusion-pde} by applying a Galerkin approach. It is well-known that standard (say $\Pone$) finite element methods yield an approximation of poor accuracy, as a consequence of the highly oscillatory nature of the problem, unless a prohibitively expensive fine mesh is employed. Dedicated multiscale approaches have thus been introduced, which provide a reasonably accurate approximation of~\eqref{eq:diffusion-pde} for a limited computational cost. Among the many multiscale approaches that have been proposed in the literature, we mention the {\em Heterogeneous Multiscale Method} (henceforth abbreviated as HMM)~\cite{abdulle_heterogeneous_2012}, the {\em Local Orthogonal Decomposition} method (LOD)~\cite{altmann_numerical_2021}, and the {\em Multiscale Finite Element Method} (MsFEM) on which we focus here (see~\cite{hou_multiscale_1997,efendiev_multiscale_2009,clb_fl_jcp} for a comprehensive exposition).

Specifically, the MsFEM method is a finite element type method, which consists of two steps:
\begin{itemize}
\item an ``offline'' phase, where highly oscillatory, problem-dependent (but right-hand side independent) basis functions are numerically computed as solutions to local problems (that mimick the reference problem on a subdomain);
\item an ``online'' phase, where a Galerkin approximation of~\eqref{eq:diffusion-pde}, performed in the finite-dimensional space generated by the basis functions computed in the offline phase, is solved.
\end{itemize}
The MsFEM approach is particularly interesting in multi-query contexts, where~\eqref{eq:diffusion-pde} is to be repeatedly solved for multiple right-hand sides $f$ (think e.g. of optimization problems, or of time-dependent problems where~\eqref{eq:diffusion-pde} would be the typical equation to solve to advance from one time step to the next). In this case, the offline phase is only performed once and a significant computational gain is thus achieved.

\smallskip

Several MsFEM variants exist, depending on the specific definition of the basis functions. Although their implementation is rightfully considered to be relatively easy (see e.g.~\cite{efendiev_multiscale_2009,schillinger_jcompphys}), they are all definitely intrusive. They indeed require to change the finite element basis set and adjust it to the problem at hand. Our aim in this article is to investigate how these approaches can be adapted (possibly at the price of a marginal loss in their efficiency) so that they become as little intrusive as possible, thereby allowing to use only a {\em legacy}, single-scale software (based on standard finite elements) to recover an accurate approximation of~\eqref{eq:diffusion-pde}. We believe that such an endeavor will eventually facilitate the dissemination of MsFEM methods within industrial, non academic codes.

\smallskip

The article is organized as follows. In Section~\ref{sec:motivation}, on the example of the diffusion problem~\eqref{eq:diffusion-pde}, we introduce the simplest MsFEM approach (namely the so-called linear MsFEM) in a Galerkin setting (see Algorithm~\ref{alg:msfem-diff}), and present its non-intrusive variant, namely Algorithm~\ref{alg:nonin-msfem-diff}. We also outline the relation between this variant and a Petrov-Galerkin variant of the MsFEM. Some elements of numerical analysis (along with an illustrative numerical result) are provided in Section~\ref{sec:analysis}, to estimate the additional error introduced by the non-intrusive implementation of the method. It is well-known that the linear MsFEM variant considered in Sections~\ref{sec:motivation} and~\ref{sec:analysis} is outperformed by several MsFEM variants. For pedagogical purposes, we have deliberately chosen to present our ideas on this simple variant and to collect in Section~\ref{sec:conc} some concluding remarks on the many possible extensions of our methodology to design non-intrusive implementations of existing approaches.

\smallskip

We postpone a comprehensive presentation of our procedure, with applications to several MsFEM variants, several equations and various boundary conditions, to an upcoming article~\cite{non_intrusif_long} (see also~\cite{biezemans_difficult_nodate}).

\section{Non-intrusive implementation of MsFEM: a simple case} \label{sec:motivation}

We provide problem~\eqref{eq:diffusion-pde} with all the usual assumptions that make it well-posed at the continuous level and amenable to a classical Galerkin approximation. In particular, we assume that $A^\eps \in L^\infty(\Omega,\bbR^{d \times d})$ satisfies the bounds
\begin{equation} \label{ass:elliptic}
  \begin{array}{c}
    \dps \forall \xi \in \bbR^d, \quad m \, |\xi|^2 \leq \xi \cdot A^\eps(x) \xi \quad \text{a.e. in $\Omega$}
    \\ \noalign{\vskip 1pt}
    \dps \text{and} \qquad \forall \xi,\eta \in \bbR^d, \quad \left| \eta \cdot A^\eps(x) \xi \right| \leq M \, |\xi| \, |\eta| \quad \text{a.e. in $\Omega$},
  \end{array}
\end{equation}
for some $M \geq m > 0$ independent of $\eps$, and that $f$ \emph{does not oscillate} on the microscopic scale $\eps$. Note that no further structural assumptions on $A^\eps$ are made (in particular, $A^\eps$ need not be of the form $A(\cdot/\eps)$ for a fixed rescaled function $A$). We respectively denote by
\begin{equation} \label{eq:def_a_F}
  a^\eps(u,v) = \int_\Omega \nabla v \cdot A^\eps \nabla u
  \qquad \text{and} \qquad
  F(v) = \int_\Omega f v
\end{equation}
the bilinear and linear forms associated to the variational formulation of~\eqref{eq:diffusion-pde}. We seek a numerical approximation of~\eqref{eq:diffusion-pde} by applying an MsFEM type Galerkin approach. To this end, we introduce a conformal simplicial mesh $\mathcal{T}_H$ of $\Omega$ (i.e., made of triangles if $d=2$, tetrahedra if $d=3$) and denote by $V_H$ the usual conformal $\Pone$ approximation space on $\mathcal{T}_H$. For any $u, v \in H^1(K)$, we also define $\dps a^\eps_K(u,v) = \int_K \nabla v \cdot A^\eps \nabla u$ and $\dps F_K(v) = \int_K f v$.

\subsection{The Multiscale Finite Element Method}

The MsFEM is a Galerkin approximation that adapts the finite-dimensional approximation space in order to obtain a satisfactory accuracy even on a coarse mesh. Let $x_1,\dots,x_{N_v}$ be the interior vertices (i.e., the $N_v$ vertices that do not lie on $\partial \Omega$) of $\mathcal{T}_H$, and let $\phiPone{i}$ be the unique element of $V_H$ such that $\phiPone{i}(x_j) = \delta_{i,j}$ for all $1 \leq j \leq N_v$. For any $1 \leq i \leq N_v$, we define the \emph{multiscale basis function} $\phiEps{i} \in H^1_0(\Omega)$, which is supported by the exact same mesh elements as $\phiPone{i}$, by
\begin{equation}
  \forall K \in \mathcal{T}_H, 
  \qquad 
  -\dive \left( A^\eps \nabla \phiEps{i} \right) = 0 \ \ \text{in $K$}
  \quad \text{and} \quad 
  \phiEps{i} = \phiPone{i} \ \ \text{on $\partial K$}.
  \label{eq:MsFEM-basis}
\end{equation}
The multiscale approximation space is defined as
$$
V_H^\eps = \operatorname{span} \left\{ \phiEps{i}, \ 1 \leq i \leq N_v \right\},
$$
and it has the same dimension as $V_H$. The MsFEM approach then consists in computing the approximation $u^\eps_H \in V_H^\eps$ defined by the problem
\begin{equation}
  \forall v^\eps_H \in V_H^\eps, \qquad a^\eps(u^\eps_H,v^\eps_H) = F(v^\eps_H).
  \label{eq:diffusion-MsFEM}
\end{equation}
Since the space $V_H^\eps$ is problem-dependent, we can hope (and this is indeed the case) the approximation~\eqref{eq:diffusion-MsFEM} to capture the exact solution much better than a $\Pone$ approximation on the same mesh (even on a mesh of size $H$ of the order of $\eps$).

\begin{remark}[Offline vs. online phase]
  In practice, the local problems~\eqref{eq:MsFEM-basis} need to be approximated, for instance using a standard $\Pone$ approximation on a fine mesh (of $K \in \mathcal{T}_H$) of mesh size $h \leq \eps$ which resolves the oscillations of $A^\eps$. We omit here this additional discretization. All what follows can readily be extended to the case when only some numerical approximation of $\phiEps{i}$ is available.

The computation of $\dps \{ \phiEps{i} \}_{1 \leq i \leq N_v}$ is called the \emph{offline phase} of the MsFEM. All the problems~\eqref{eq:MsFEM-basis} on different mesh elements are independent of each other, and can thus be solved in parallel. On the other hand, the term \emph{online phase} is used for the resolution of the global problem~\eqref{eq:diffusion-MsFEM}, where the number of degrees of freedom is the same as in a standard $\Pone$ approximation on $V_H$. \qed
\end{remark}

The MsFEM approach can schematically be presented as Algorithm~\ref{alg:msfem-diff}. Lines~\ref{alg:msfem-diff-offl1}--\ref{alg:msfem-diff-offl2} (resp.~\ref{alg:msfem-diff-onl1}--\ref{alg:msfem-diff-onl2}) constitute the offline (resp. online) phase.

\begin{algorithm}[ht] % [ht] permet de mettre l'algo dans la page et qu'il n'occupe pas a lui tout seul une page
\caption{MsFEM approach for Problem~\eqref{eq:diffusion-pde} (see comments in the text)}
\label{alg:msfem-diff}
\begin{algorithmic}[1]
  \State Construct a mesh $\mathcal{T}_H$ of $\Omega$, denote $N_v$ the number of internal vertices and $\mathcal{N}(n,K)$ the global index of the vertex of $K \in \mathcal{T}_H$ that has local index $1 \leq n \leq d+1$ in $K$
  \label{alg:msfem-diff-offl1}
  \medskip
  \For{$1 \leq i \leq N_v$}
  \State Solve for $\phiEps{i}$ in~\eqref{eq:MsFEM-basis} 
  \EndFor
  \medskip
  %  \Comment{Construct the matrix and right-hand side for the linear system associated to~\eqref{eq:diffusion-MsFEM}:}
  \State Set $\mathds{A}^\eps := 0$ and $\mathds{F}^\eps := 0$
  \ForAll{$K \in \mathcal{T}_H$}
  \For{$1 \leq m \leq d+1$}
  \State Set $j := \mathcal{N}(m,K)$  
  \For{$1 \leq n \leq d+1$}
  \State Set $i := \mathcal{N}(n,K)$ and $\mathds{A}^\eps_{j,i} \ \ \text{+=} \ \ a^\eps_K\left(\phiEps{i},\phiEps{j}\right)$
  \label{alg:msfem-diff-stiffness}
  \EndFor
  \label{alg:msfem-diff-offl2}
  \State Set $\mathds{F}^\eps_j \ \ \text{+=} \ \ F_K\left(\phiEps{j}\right)$
  \label{alg:msfem-diff-onl1}
  \EndFor
  \EndFor
  \medskip
  \State Solve the linear system $\mathds{A}^\eps \, U^\eps = \mathds{F}^\eps$
  \State Obtain the MsFEM approximation $\dps u^\eps_H = \sum_{i=1}^{N_v} U_i^\eps \, \phiEps{i}$
  \label{alg:msfem-diff-onl2}
\end{algorithmic}
\end{algorithm}

\smallskip

Implementing Algorithm~\ref{alg:msfem-diff} in an industrial code is challenging. Indeed, the practical implementation of any finite element method relies on (i) the construction of a mesh, (ii) the construction of the linear system associated to the discrete variational formulation and (iii) the resolution of the linear system. An efficient implementation of the second step heavily relies on the choice of the discretization space. Regarding the construction of the linear system (performed in line~\ref{alg:msfem-diff-stiffness} of Algorithm~\ref{alg:msfem-diff}), it is by no means obvious to adapt existing finite element codes based on generic spaces like $V_H$ to a different, problem-dependent choice of space such as $V_H^\eps$. No analytic expressions for the basis functions $\phiEps{i}$ are available (and thus a fine mesh should be used to approximate them), the computation of $\dps a^\eps_K\left(\phiEps{i},\phiEps{j}\right)$ and $\dps F_K\left(\phiEps{j}\right)$ should be performed by quadrature rules on the fine mesh because the integrands are highly oscillatory, one should have at hand the correspondence between element and vertex indices in the coarse mesh, the assembly of the global stiffness matrix $\dps \left\{ \mathds{A}^\eps_{j,i} \right\}_{1 \leq i,j \leq N_v}$ should be manually performed, \dots To alleviate these obstacles, we shall next introduce a way of implementing MsFEM that capitalizes on an existing code for solving~\eqref{eq:diffusion-pde} by a $\Pone$ approximation on $\mathcal{T}_H$ in the case of slowly-varying diffusion coefficients.

To the best of our knowledge, the question of how to make MsFEM approaches less intrusive has not been studied in the literature, and this work (that will be complemented in the future works~\cite{non_intrusif_long,biezemans_difficult_nodate}) is a first step in that direction. On the other hand, for some other multiscale approaches (including HMM and LOD), this question has been (at least partially) addressed. By construction, HMM methods are less invasive, since they primarily aim at approximating $u^\eps$ on the {\em coarse scales}. The first step of these methods somewhat consists in building an effective, slowly-varying diffusion coefficient (which plays the role of the matrix $\overline{A}$ introduced in~\eqref{eq:linear-system-effective} below), which can next be used in any single-scale solver. The LOD approach, which, similarly to MsFEM, aims at approximating $u^\eps$ on {\em the coarse and the fine scales}, is also invasive in general, since it also introduces adapted basis functions. As shown in~\cite{gallistl_computation_2017}, the LOD can be recast as a $\Pone$ discretization of an appropriate single-scale problem, an observation which opens the way to non-intrusive implementations. Note however that, despite its current intrusiveness, MsFEM has its own advantages over other multiscale approaches: it directly aims at approximating the oscillatory solution $u^\eps$ ({\em including} its fine scale details), and makes use of {\em fully localized} basis functions to do so.

\subsection{Equivalent problem on the macroscopic scale} \label{sec:equivalent}

Our starting point for reducing intrusiveness in the above MsFEM implementation is the following key observation. On any $K \in \mathcal{T}_H$, by linearity of the definition~\eqref{eq:MsFEM-basis} of $\phiEps{i}$ in terms of $\phiPone{i}$, and because the finite element space $V_H$ consists of functions that are piecewise affine (that is to say, $\nabla \phiPone{i}$ is constant in each $K$), we have the expansion
\begin{equation} \label{eq:fondamental}
  \left. \phiEps{i}(x) \right\vert_K = \phiPone{i}(x) + \sum_{\alpha=1}^d (\partial_\alpha \phiPone{i})\vert_K \ \VK{\alpha}(x),
\end{equation}
for any basis function $\phiEps{i}$ of $V^\eps_H$. Here, $\VK{\alpha} \in H^1_0(K)$ is defined as the solution to the local problem
\begin{equation}
  -\dive \left(A^\eps \nabla \VK{\alpha}\right) = \dive \left(A^\eps e_\alpha\right) \ \ \text{in $K$},
  \qquad
  \VK{\alpha} = 0 \ \ \text{on $\partial K$},
  \label{eq:MsFEM-Vxy}
\end{equation}
where $e_\alpha$ denotes the $\alpha$-th canonical unit vector of $\bbR^d$. Considering indeed the right-hand side of~\eqref{eq:fondamental}, we compute that its gradient in $K$ is $\dps \sum_{\alpha=1}^d (\partial_\alpha \phiPone{i})\vert_K \ \left(e_\alpha + \nabla \VK{\alpha}(x)\right)$. In view of~\eqref{eq:MsFEM-Vxy}, we obtain that $\dps -\dive A^\eps \left[ \sum_{\alpha=1}^d (\partial_\alpha \phiPone{i})\vert_K \ \left(e_\alpha + \nabla \VK{\alpha}\right) \right] = 0$ in $K$. The right-hand side of~\eqref{eq:fondamental} therefore satisfies the PDE in~\eqref{eq:MsFEM-basis}. It also satisfies the boundary conditions in~\eqref{eq:MsFEM-basis}, in view of the boundary conditions in~\eqref{eq:MsFEM-Vxy}. Since the solution to~\eqref{eq:MsFEM-basis} is unique, we deduce the identity~\eqref{eq:fondamental}.

\smallskip

For each $K \in \mathcal{T}_H$ and each $1 \leq \alpha \leq d$, we extend $\VK{\alpha}$ by $0$ outside $K$, thereby obtaining a function $\VK{\alpha} \in H^1_0(\Omega)$. We then deduce from~\eqref{eq:fondamental} that there is a one-to-one correspondence between functions in $V_H^\eps$ and functions in $V_H$. More precisely, for any $v^\eps_H \in V_H^\eps$, there exists a unique $v_H \in V_H$ such that
\begin{equation}
  v_H^\eps = v_H + \sum_{K \in \mathcal{T}_H} \sum_{\alpha=1}^d (\partial_\alpha v_H) \vert_K \ \VK{\alpha},
  \label{eq:expansion-MS-P1}
\end{equation}
and conversely, for any $v_H \in V_H$, we have $\dps v_H + \sum_{K \in \mathcal{T}_H} \sum_{\alpha=1}^d (\partial_\alpha v_H) \vert_K \ \VK{\alpha} \in V_H^\eps$. In particular, the function associated to $\phiEps{i} \in V_H^\eps$ is $\phiPone{i} \in V_H$.

\smallskip

We now consider the linear system associated to~\eqref{eq:diffusion-MsFEM} and insert therein the expansion~\eqref{eq:expansion-MS-P1} for the multiscale basis functions. The solution to~\eqref{eq:diffusion-MsFEM} reads $\dps u^\eps_H = \sum_{i=1}^{N_v} U^\eps_i \, \phiEps{i}$, where $\dps U^\eps = \left(U^\eps_1, \dots, U^\eps_{N_v}\right)^T$ is the solution to
\begin{equation}
  \mathds{A}^\eps \, U^\eps = \mathds{F}^\eps,
  \label{eq:discrete-ms}
\end{equation}
with
\begin{equation*}
  \forall 1 \leq i, j \leq N_v,
  \qquad
  \mathds{A}^\eps_{j,i} = a^\eps\left(\phiEps{i},\phiEps{j}\right),
  \qquad 
  \mathds{F}^\eps_j = F\left(\phiEps{j}\right).  
%  \label{eq:linear-system-ms}
\end{equation*}
Taking $v_H^\eps = \phiEps{i}$ (and thus $v_H = \phiPone{i}$) in~\eqref{eq:expansion-MS-P1} and using that $\nabla \phiPone{i}$ is piecewise constant, we write
$$
\nabla \phiEps{i} = \sum_{K \in \mathcal{T}_H} \sum_{\alpha=1}^d (\partial_\alpha \phiPone{i})\vert_K \ \left(e_\alpha + \nabla \VK{\alpha}\right).
$$
Inserting this relation in the definition of the matrix elements $\mathds{A}^\eps_{j,i}$, we obtain
\begin{align}
  \mathds{A}^\eps_{j,i} 
  &=
  \int_\Omega \nabla \phiEps{j} \cdot A^\eps \nabla \phiEps{i}
  \nonumber
  \\
  &=
  \sum_{K \in \mathcal{T}_H} \sum_{\alpha,\beta=1}^d (\partial_\beta \phiPone{j})\vert_K \ \left( \int_K \left(e_\beta + \nabla \VK{\beta}\right) \cdot A^\eps \left(e_\alpha + \nabla \VK{\alpha}\right) \right) \ (\partial_\alpha \phiPone{i})\vert_K
  \nonumber
  \\
  &=
  \sum_{K \in \mathcal{T}_H} \sum_{\alpha,\beta=1}^d (\partial_\beta \phiPone{j})\vert_K \ a^\eps_K\left(x^\alpha + \VK{\alpha},x^\beta + \VK{\beta}\right) \ (\partial_\alpha \phiPone{i})\vert_K,
  \label{eq:stiffness-diffusion-ms}
\end{align}
where $x^\alpha = x \cdot e_\alpha$ is the $\alpha$-th coordinate. We now define the piecewise constant matrix-valued field $\overline{A} \in \mathbb{P}_0(\mathcal{T}_H,\bbR^{d\times d})$ by
\begin{equation}
  \left. \overline{A}_{\beta,\alpha} \right\vert_K 
  =
  \frac{1}{|K|} \, a^\eps_K\left(x^\alpha + \VK{\alpha},x^\beta + \VK{\beta}\right)
  \quad \text{for each $K \in \mathcal{T}_H$ and $1 \leq \alpha,\beta \leq d$},
  \label{eq:linear-system-effective}
\end{equation}
where $|K|$ denotes the area or volume of the mesh element $K$. Using~\eqref{eq:MsFEM-Vxy} and the bounds~\eqref{ass:elliptic} satisfied by $A^\eps$, it is easy to show (see~\cite{non_intrusif_long}) that $\overline{A}$ satisfies the following uniform lower and upper bounds: for any $\xi$ and $\eta$ in $\bbR^d$,
\begin{equation}
  m \, |\xi|^2 \leq \xi \cdot \overline{A}(x) \xi \ \ \text{and} \ \ \left| \eta \cdot \overline{A}(x) \xi \right| \leq M \left( 1 + \frac{M}{m} \right) |\xi| \, |\eta| \ \ \text{a.e. in $\Omega$}.
  \label{ass:elliptic_Abar}
\end{equation}
Motivated by~\eqref{eq:stiffness-diffusion-ms} and~\eqref{eq:linear-system-effective}, we introduce the coarse-scale problem
\begin{equation}
  -\dive \left(\overline{A} \, \nabla u \right) = f \ \ \text{in $\Omega$}, \qquad u = 0 \ \ \text{on $\partial \Omega$},
  \label{eq:diffusion-pde-macro}
\end{equation}
and its $\Pone$ Galerkin discretization: find $u_H \in V_H$ such that
\begin{equation}
  \forall v_H \in V_H, \qquad a^{\overline{A}}(u_H,v_H) = F(v_H),
  \label{eq:diffusion-vf-macro}
\end{equation}
where the linear form $F$ is defined by~\eqref{eq:def_a_F} and the bilinear form $a^{\overline{A}}$ is defined by
\begin{equation*}
  \forall u, v \in H^1_0(\Omega), \qquad a^{\overline{A}}(u,v) = \sum_{K \in \mathcal{T}_H} a^{\overline{A}}_K(u,v) \quad \text{with} \quad a^{\overline{A}}_K(u,v) = \int_K \nabla v \cdot \overline{A} \nabla u. 
\end{equation*}
Problem~\eqref{eq:diffusion-vf-macro} equivalently writes
\begin{equation}
  \mathds{A}^{\Pone} \, U^{\Pone} = \mathds{F}^{\Pone},
  \label{eq:discrete-macro}
\end{equation}
with
\begin{equation}
  \forall 1 \leq i, j \leq N_v,
  \qquad
  \mathds{A}^{\Pone}_{j,i} = a^{\overline{A}}\left(\phiPone{i},\phiPone{j}\right),
  \qquad 
  \mathds{F}^{\Pone}_j = F\left(\phiPone{j}\right).  
  \label{eq:linear-system-macro}
\end{equation}
We then deduce from~\eqref{eq:stiffness-diffusion-ms} that
\begin{equation}
  \mathds{A}^\eps_{j,i} = \int_\Omega \nabla \phiPone{j} \cdot \overline{A} \nabla \phiPone{i} = \mathds{A}^{\Pone}_{j,i},
  \label{eq:linear-system-P1}
\end{equation}
where we recall that the piecewise constant matrix $\overline{A}$ given by~\eqref{eq:linear-system-effective}, and therefore the stiffness matrix $\mathds{A}^{\Pone}$, depends on the fine scale oscillations of $A^\eps$. The above calculations yield the following result.

\begin{lemma} \label{lem:fondamental}
The stiffness matrix $\mathds{A}^\eps$ in~\eqref{eq:discrete-ms} of the MsFEM problem~\eqref{eq:diffusion-MsFEM} is identical to the stiffness matrix $\mathds{A}^{\Pone}$ in~\eqref{eq:discrete-macro} of the $\Pone$ problem~\eqref{eq:diffusion-vf-macro}. 
\end{lemma}

We note that the right-hand vector $\mathds{F}^\eps$ in~\eqref{eq:discrete-ms} is in general different from the right-hand side vector $\mathds{F}^{\Pone}$ in~\eqref{eq:discrete-macro}, since we integrate $f$ against highly oscillatory basis functions in the former problem and against $\Pone$ basis functions in the latter. The solutions $U^\eps$ and $U^{\Pone}$ to~\eqref{eq:discrete-ms} and~\eqref{eq:discrete-macro}, respectively, are thus a priori different.

\smallskip

The above observations suggest to use the identity~\eqref{eq:linear-system-P1} of the stiffness matrices to replace the MsFEM discrete problem~\eqref{eq:discrete-ms} by the discrete problem~\eqref{eq:discrete-macro} stemming from the $\Pone$ Galerkin approximation of the {\em single scale} problem~\eqref{eq:diffusion-pde-macro}, which itself may be easily implemented in legacy codes. This results in our non-intrusive MsFEM strategy presented in Algorithm~\ref{alg:nonin-msfem-diff}. We can distinguish there the \emph{offline phase} (in lines~\ref{alg:nonin-msfem-diff-offl1}--\ref{alg:nonin-msfem-diff-offl2}) and the \emph{online phase} (in lines~\ref{alg:nonin-msfem-diff-onl1}--\ref{alg:nonin-msfem-diff-onl2}). The few lines that differ between Algorithm~\ref{alg:nonin-msfem-diff} and a standard $\Pone$ algorithm are highlighted in blue. The lines in black are already present in standard codes and are written in Algorithm~\ref{alg:nonin-msfem-diff} for the sake of completeness.

\begin{algorithm}[ht] % [ht] permet de mettre l'algo dans la page et qu'il n'occupe pas a lui tout seul une page
  \caption{Non-intrusive MsFEM approach for Problem~\eqref{eq:diffusion-pde} (see comments in the text)}
  \label{alg:nonin-msfem-diff}
  \begin{algorithmic}[1]
    \State Construct a mesh $\mathcal{T}_H$ of $\Omega$, denote $N_v$ the number of internal vertices and $\mathcal{N}(n,K)$ the global index of the vertex of $K \in \mathcal{T}_H$ that has local index $1 \leq n \leq d+1$ in $K$
    \label{alg:nonin-msfem-diff-offl1}
    \medskip
    \color{blue}
    \ForAll{$K \in \mathcal{T}_H$}
    \For{$1 \leq \alpha \leq d$}
    \label{debut_VK}
    \State Solve for $\VK{\alpha}$ defined by~\eqref{eq:MsFEM-Vxy}
    \EndFor
    \State Compute $\overline{A} \vert_K$ defined by~\eqref{eq:linear-system-effective}
    \label{fin_VK}    
    \EndFor
    \color{black}
    \medskip
    \State Set $\mathds{A}^{\Pone} := 0$ and $\mathds{F}^{\Pone} := 0$
    \ForAll{$K \in \mathcal{T}_H$}
    \For{$1 \leq m \leq d+1$}
    \State Set $j := \mathcal{N}(m,K)$  
    \For{$1 \leq n \leq d+1$}
    \State Set $i := \mathcal{N}(n,K)$ and $\mathds{A}^{\Pone}_{j,i} \ \ \text{+=} \ \ a^{\overline{A}}_K\left(\phiPone{i},\phiPone{j}\right)$ 
    \EndFor
    \label{alg:nonin-msfem-diff-offl2}
    \State Compute $\mathds{F}^{\Pone}_j \ \ \text{+=} \ \ F_K\left(\phiPone{j}\right)$
    \label{alg:nonin-msfem-diff-onl1}
    \EndFor
    \EndFor
    \medskip
    \State Solve the linear system $\mathds{A}^{\Pone} U^{\Pone} = \mathds{F}^{\Pone}$
    \State Obtain the coarse approximation $\dps u_H = \sum_{i=1}^{N_v} U_i^{\Pone} \, \phiPone{i}$
    \color{blue}
    \State Obtain the MsFEM approximation given in $K$ by $\dps \overline{u}^\eps_H = u_H + \sum_{\alpha=1}^d (\partial_\alpha u_H)\vert_K \ \VK{\alpha}$
    \color{black}
    \label{alg:nonin-msfem-diff-onl2}
  \end{algorithmic}
\end{algorithm}

\smallskip

The superiority of Algorithm~\ref{alg:nonin-msfem-diff} over the classical MsFEM Algorithm~\ref{alg:msfem-diff} is that the global problem of the online phase (including its right-hand side) can be completely constructed and solved using a pre-existing $\Pone$ PDE solver. The only requirements in the legacy code are the ability to provide piecewise constant diffusion coefficients to the solver and the existence of a procedure which provides the value of the solution at any point in $\Omega$. The part of the offline phase which manipulates fine meshes (in lines~\ref{debut_VK}-\ref{fin_VK}) and the post-processing phase (in line~\ref{alg:nonin-msfem-diff-onl2}) can, on the other hand, be developed independently. The requirement for these fine-scale solvers is that they have access to the coarse mesh $\mathcal{T}_H$ of the global solver and that they can evaluate $u_H$ anywhere in $\Omega$ (which is useful in line~\ref{alg:nonin-msfem-diff-onl2}). Note also that the fine-scale problem~\eqref{eq:MsFEM-Vxy} is only indexed by the coarse mesh element $K$, in contrast to the fine scale problem~\eqref{eq:MsFEM-basis}, which is indexed both by the coarse mesh element $K$ and the vertex index $i$. In the latter case, one has to know, for each element $K$, the global number of the element vertices, a piece of information which may be difficult to access to in a legacy code. In the former case, this correspondence is not needed to compute $\overline{A}$ and $\overline{u}^\eps_H$, which are entirely defined element-wise.

\begin{remark}[Link with homogenization theory]
  The above non-intrusive implementation involves quantities which are of course reminiscent of standard quantities introduced in homogenization. The problem~\eqref{eq:MsFEM-Vxy}, the effective diffusion matrix~\eqref{eq:linear-system-effective}, the single-scale problem~\eqref{eq:diffusion-pde-macro} and the expansion~\eqref{eq:expansion-MS-P1} ressemble the corrector problem, the homogenized matrix, the homogenized problem and the two-scale expansion of the oscillatory solution, respectively. In the same spirit, in the case when $A^\eps$ is the rescaling of some periodic matrix, it can be shown (see~\cite{non_intrusif_long}) that the effective diffusion matrix $\overline{A}$ converges to the homogenized matrix when $\eps \to 0$. \qed  
\end{remark}  

As shown by~\eqref{eq:linear-system-P1}, the matrix $\mathds{A}^{\Pone}$ in~\eqref{eq:discrete-macro} is identical to the matrix $\mathds{A}^\eps$ in~\eqref{eq:discrete-ms}. However, in general, the right-hand sides $\mathds{F}^\eps$ in~\eqref{eq:discrete-ms} and $\mathds{F}^{\Pone}$ in~\eqref{eq:discrete-macro} are different. This motivates the introduction of the following Petrov-Galerkin variant of the MsFEM: find $u^{\eps,PG}_H \in V^\eps_H$ such that 
\begin{equation}
  \forall v_H \in V_H, \qquad a^\eps\left(u^{\eps,PG}_H,v_H\right) = F(v_H).
    \label{eq:diffusion-MsFEM-testP1}
\end{equation}
Note that, in contrast to~\eqref{eq:diffusion-MsFEM}, we take the test functions in the $\Pone$ space $V_H$ rather than in the multiscale space $V_H^\eps$. We denote $\dps \mathds{A}^{\eps,PG}_{j,i} = a^\eps\left(\phiEps{i},\phiPone{j}\right)$ the stiffness matrix of the resulting linear system. Of course, the right-hand side vector of this linear system is equal to $\mathds{F}^{\Pone}$ defined by~\eqref{eq:linear-system-macro}.

\begin{lemma} \label{lem:MsFEMP1-macro-diff}
  The stiffness matrices $\mathds{A}^\eps$ and $\mathds{A}^{\eps,PG}$ of the problems~\eqref{eq:diffusion-MsFEM} and~\eqref{eq:diffusion-MsFEM-testP1} are identical (and thus identical to $\mathds{A}^{\Pone}$). The right-hand side vector of~\eqref{eq:diffusion-MsFEM-testP1} is identical to the right-hand side vector $\mathds{F}^{\Pone}$ of the problem~\eqref{eq:diffusion-vf-macro}. The solution $u^{\eps,PG}_H$ to the Petrov-Galerkin MsFEM~\eqref{eq:diffusion-MsFEM-testP1} thus coincides with the function $\overline{u}^\eps_H$ computed by Algorithm~\ref{alg:nonin-msfem-diff}.
\end{lemma}

\begin{proof}
  To prove this lemma, it is enough to show that the stiffness matrices of~\eqref{eq:diffusion-MsFEM} and~\eqref{eq:diffusion-MsFEM-testP1} are equal. Using an integration by parts, we compute
\begin{align}
  \! \! a^\eps\left(\phiEps{i},\phiEps{j} - \phiPone{j}\right)
  &=
  \sum_{K\in\mathcal{T}_H} \int_K \nabla \left( \phiEps{j} - \phiPone{j} \right) \cdot A^\eps \nabla \phiEps{i}
  \nonumber
  \\
  &=
  \sum_{K\in\mathcal{T}_H} \int_{\partial K} \! \! \left( \phiEps{j} - \phiPone{j} \right) n \cdot A^\eps \nabla \phiEps{i} - \int_K \! \! \left( \phiEps{j} - \phiPone{j} \right) \dive \left( A^\eps \nabla \phiEps{i} \right),
  \label{eq:IPP-bubbles}
\end{align}
where $n$ is the unit outward normal vector on $\partial K$. The two terms above vanish in view of~\eqref{eq:MsFEM-basis}. This implies the identity of the stiffness matrices of~\eqref{eq:diffusion-MsFEM} and~\eqref{eq:diffusion-MsFEM-testP1} and the well-posedness of the Petrov-Galerkin approximation~\eqref{eq:diffusion-MsFEM-testP1}.
\end{proof}

To summarize, the above procedure to go from the MsFEM problem~\eqref{eq:diffusion-MsFEM} to its non-intrusive implementation described in Algorithm~\ref{alg:nonin-msfem-diff} is based on the following steps:
\begin{enumerate}
\item we use the linearity of the problem and the fact that gradients of $\Pone$ basis functions are constant in each mesh element to establish the identity~\eqref{eq:fondamental};
\item we can then recast the MsFEM stiffness matrix $\mathds{A}^\eps$ as the stiffness matrix $\mathds{A}^{\Pone}$ of the $\Pone$ discretization of an appropriate problem;
\item we approximate the right-hand side $\mathds{F}^\eps$ of the MsFEM problem by a right-hand side $\mathds{F}^{\Pone}$ which can be computed in a manner consistent with a $\Pone$ discretization;
\item we postprocess the $\Pone$ solution to obtain an approximation of the reference solution.  
\end{enumerate}
We nowhere use in these steps that the discrete problem we are actually solving (here~\eqref{eq:discrete-macro}) in fact corresponds to some discretization (here, a Petrov-Galerkin discretization) of the reference problem. This correspondence is useful to estimate the error introduced by the non-intrusive implementation (a task we perform in Section~\ref{sec:analysis} below), but it does not necessarily hold for other variants of MsFEM (e.g. the oversampling variant briefly mentioned in Section~\ref{sec:conc} below) and it is not required to put the non-intrusive approach in action.

We expect the computational cost of Algorithm~\ref{alg:nonin-msfem-diff} to be smaller than that of Algorithm~\ref{alg:msfem-diff}, since $\mathds{F}^{\Pone}$ is cheaper to compute than $\mathds{F}^\eps$ (we do not need to use a quadrature rule on the fine mesh). On the other hand, Algorithm~\ref{alg:nonin-msfem-diff} (and non-intrusive implementations in general) may introduce additional numerical errors. We estimate these in Section~\ref{sec:analysis}.

\section{Comparison of Galerkin and Petrov-Galerkin MsFEM}\label{sec:analysis}

In this section, we estimate the difference between the solutions to the Galerkin and the Petrov-Galerkin approximations~\eqref{eq:diffusion-MsFEM} and~\eqref{eq:diffusion-MsFEM-testP1}. We first recall that, for any $v^\eps_H \in V_H^\eps$, there exists a unique $v_H \in V_H$ such that~\eqref{eq:expansion-MS-P1} holds. We now establish a variational relation between $v_H^\eps$ and $v_H$.

\begin{lemma} \label{lem:expansion-MS-P1-char}
Let $v^\eps_H \in V_H^\eps$. There exists a unique $v_H \in V_H$ such that~\eqref{eq:expansion-MS-P1} holds, and $v_H$ is the unique solution in $V_H$ to the problem
\begin{equation}   \label{eq:expansion-MS-P1-char1}
  \forall w_H \in V_H, \qquad a^{\overline{A}}(v_H,w_H) = a^\eps(v^\eps_H,w_H),
\end{equation}
where $\overline{A}$ is defined by~\eqref{eq:linear-system-effective}. In addition, we have $\dps \| \nabla v_H \|_{L^2(\Omega)} \leq \frac{M}{m} \, \left\| \nabla v_H^\eps \right\|_{L^2(\Omega)}$.
\end{lemma}

\begin{proof}
We take some $v^\eps_H \in V_H^\eps$ and expand it following~\eqref{eq:expansion-MS-P1}: we thus write $\dps v_H^\eps = v_H + \sum_{K \in \mathcal{T}_H} \sum_{\alpha=1}^d (\partial_\alpha v_H) \vert_K \ \VK{\alpha}$ for some $v_H \in V_H$. Consider now some $w_H \in V_H$. Using~\eqref{eq:MsFEM-Vxy}, \eqref{eq:linear-system-effective} and that $\nabla w_H$ and $\nabla v_H$ are piecewise constant, we compute that
$$
a^\eps(v^\eps_H,w_H)
=
\sum_{K \in \mathcal{T}_H} \sum_{\alpha,\beta=1}^d (\partial_\beta w_H)\vert_K \ \left( \int_K \! e_\beta \cdot A^\eps \left(e_\alpha + \nabla \VK{\alpha}\right) \right) \ (\partial_\alpha v_H)\vert_K
=
a^{\overline{A}}(v_H,w_H).
$$
We thus observe that $v_H$ satisfies~\eqref{eq:expansion-MS-P1-char1}. In addition, the equation~\eqref{eq:expansion-MS-P1-char1} completely characterizes $v_H$, since $a^{\overline{A}}$ is coercive on $H^1_0(\Omega)$ in view of~\eqref{ass:elliptic_Abar}. The estimate on $v_H$ directly follows by taking $w_H = v_H$ in~\eqref{eq:expansion-MS-P1-char1} and using the coercivity of $\overline{A}$ and the upper bound on $A^\eps$.
\end{proof}

We now proceed with the following error estimate.

\begin{lemma} \label{lem:diff-MsFEM-PG-estimate}
Let $u^{\eps,G}_H$ denote the solution to the approximation~\eqref{eq:diffusion-MsFEM} (provided by Algorithm~\ref{alg:msfem-diff}), where we have added the superscript $G$ to emphasize that~\eqref{eq:diffusion-MsFEM} is a {\em Galerkin} approximation. Let $u^{\eps,PG}_H$ be the solution to the Petrov-Galerkin appproximation~\eqref{eq:diffusion-MsFEM-testP1} (provided by Algorithm~\ref{alg:nonin-msfem-diff}). Assume that $f \in L^2(\Omega)$. There exists a constant $C$ independent of $\eps$, $H$ and $f$ such that
\begin{equation*}
  \left\| u^{\eps,G}_H - u^{\eps,PG}_H \right\|_{H^1(\Omega)} \leq C \, H \, \| f \|_{L^2(\Omega)}.
\end{equation*}
\end{lemma}

The classical error estimate for the Galerkin MsFEM approach~\eqref{eq:diffusion-MsFEM} is obtained in the literature under the assumption that $A^\eps$ is actually the rescaling of a given periodic matrix. The bound in the estimate reads $C \, (H + \sqrt{\eps} + \sqrt{\eps/H})$ for some $C$ independent of $\eps$ and $H$ (see e.g.~\cite{efendiev_multiscale_2009}). Lemma~\ref{lem:diff-MsFEM-PG-estimate} and a triangle inequality show that the same convergence rate holds for the Petrov-Galerkin MsFEM approach~\eqref{eq:diffusion-MsFEM-testP1}, of course under the same periodicity assumption.

\begin{proof}
Let $e^\eps_H = u^{\eps,G}_H - u^{\eps,PG}_H$. Since the numerical approximations $u^{\eps,G}_H$ and $u^{\eps,PG}_H$ both belong to $V_H^\eps$, we are in position to use~\eqref{eq:expansion-MS-P1} for $e^\eps_H \in V_H^\eps$: there thus exists $e^\Pone_H \in V_H$ such that
\begin{equation} \label{eq:toto2}
  e^\eps_H = e^\Pone_H + e^{\rm osc}_H \quad \text{with} \quad e^{\rm osc}_H = \sum_{K\in\mathcal{T}_H} \sum_{\alpha=1}^d (\partial_\alpha e^\Pone_H)\vert_K \ \VK{\alpha}.
\end{equation}
We can thus write
\begin{equation*}
  a^\eps(e^\eps_H,e^\eps_H)
  =
  a^\eps\left(u^{\eps,G}_H,e^\eps_H\right) -
  a^\eps\left(u^{\eps,PG}_H,e^\Pone_H\right) -
  a^\eps\left(u^{\eps,PG}_H,e^{\rm osc}_H\right).
\end{equation*}
Since $e^\eps_H$ can be used as a test function in~\eqref{eq:diffusion-MsFEM} and $e^\Pone_H$ in~\eqref{eq:diffusion-MsFEM-testP1}, this implies
$$
%\begin{equation}
  a^\eps(e^\eps_H,e^\eps_H)
  =
  (f,e^\eps_H)_{L^2(\Omega)} - (f,e^\Pone_H)_{L^2(\Omega)} - a^\eps(u^{\eps,PG}_H,e^{\rm osc}_H)
  =
  (f,e^{\rm osc}_H)_{L^2(\Omega)} - a^\eps(u^{\eps,PG}_H,e^{\rm osc}_H).
  %\label{eq:MsFEM-PG-estimate0}
%\end{equation}
$$
Using that $e^{\rm osc}_H$ vanishes on all the edges of the coarse mesh (because of the boundary conditions satisfied by $\VK{\alpha}$), the same integration by parts that led to~\eqref{eq:IPP-bubbles} shows that $\dps a^\eps\left(u^{\eps,PG}_H,e^{\rm osc}_H\right) = 0$. We therefore infer from the above equation and the Cauchy-Schwarz inequality that
\begin{equation} \label{eq:toto}
  a^\eps(e^\eps_H,e^\eps_H) \leq \| f \|_{L^2(\Omega)} \, \left\| e^{\rm osc}_H \right\|_{L^2(\Omega)}.
\end{equation}
We proceed by bounding $\| e^{\rm osc}_H \|_{L^2(\Omega)}$ from above in terms of $\| \nabla e^\eps_H \|_{L^2(\Omega)}$ in the right-hand side of~\eqref{eq:toto}. The Poincaré inequality in each $K$ implies that there exists a constant $C$, independent of $H$ but dependent on the regularity of the mesh, such that, for any $K \in \mathcal{T}_H$,
\begin{equation}
    \left\| e^{\rm osc}_H \right\|_{L^2(K)} \leq C \, H \, \left\| \nabla e^{\rm osc}_H \right\|_{L^2(K)}.
    \label{eq:MsFEM-PG-estimate1}
\end{equation}
Using the problem~\eqref{eq:MsFEM-Vxy} satisfied by each $\VK{\alpha}$, the identity~\eqref{eq:toto2} and the fact that $\partial_\alpha e^\Pone_H$ is constant in each $K$, we observe that $e^{\rm osc}_H$ satisfies the following variational formulation in each $K\in\mathcal{T}_H$:
\begin{equation*}
  \forall v \in H^1_0(K), \qquad
  a^\eps_K\left(e^{\rm osc}_H,v\right) 
  =
  -\sum_{\alpha=1}^d (\partial_\alpha e^\Pone_H)\vert_K \ a^\eps_K\left(x^\alpha,v\right) 
  = 
  -a^\eps_K\left(e^\Pone_H,v\right).
\end{equation*}
Upon testing against $v = e^{\rm osc}_H \in H^1_0(K)$ and using the bounds~\eqref{ass:elliptic}, it follows that $\dps \left\| \nabla e^{\rm osc}_H \right\|_{L^2(K)} \leq \frac{M}{m} \, \left\| \nabla e^\Pone_H \right\|_{L^2(K)}$. We thus deduce from~\eqref{eq:MsFEM-PG-estimate1} that
\begin{equation}
  \left\| e^{\rm osc}_H \right\|_{L^2(\Omega)}
  \leq 
  C \, H \ \frac{M}{m} \, \left\| \nabla e^\Pone_H \right\|_{L^2(\Omega)}.
  \label{eq:MsFEM-PG-estimate3}
\end{equation}
In order to estimate $\nabla e^\Pone_H$ in terms of $e^\eps_H$, we apply Lemma~\ref{lem:expansion-MS-P1-char} to $v_H^\eps = e^\eps_H \in V_H^\eps$ and obtain that $\dps \left\| \nabla e^\Pone_H \right\|_{L^2(\Omega)} \leq \frac{M}{m} \, \left\| \nabla e^\eps_H \right\|_{L^2(\Omega)}$. Upon inserting this estimate and~\eqref{eq:MsFEM-PG-estimate3} in the right-hand side of~\eqref{eq:toto}, and using the lower bound~\eqref{ass:elliptic} on $A^\eps$ in the left-hand side of~\eqref{eq:toto}, it follows that $\dps \| \nabla e^\eps_H \|_{L^2(\Omega)} \leq C \, H \, \| f \|_{L^2(\Omega)}$. We conclude the proof by the application of a Poincaré inequality in $\Omega$.
\end{proof}

We now present some numerical examples to illustrate the above result. We take $\Omega = (0,1)^2$, $A^\eps(x) = a^\eps(x) \, \text{Id}$ with $a^\eps(x) = 1 + 100 \, \cos^2(\pi \, x_1 /\eps) \, \sin^2(\pi \, x_2 /\eps)$ (which is $\eps$-periodic) and $f(x) = \sin(x_1) \, \cos(x_2)$. For $\eps = \pi/150 \approx 0.02$, we consider the reference solution $u_{\rm ref}$ (computed in practice on a fine mesh of size $h=1/1024$), and, for various values of $H$, the solution $u^{\eps,G}_H$ to the Galerkin approximation~\eqref{eq:diffusion-MsFEM} and the solution $u^{\eps,PG}_H$ to the Petrov-Galerkin approximation~\eqref{eq:diffusion-MsFEM-testP1}. In Table~\ref{tab:LinGal-vs-LinPG}, we show the error $\dps \left\| u^{\eps,G}_H - u_{\rm ref} \right\|_{H^1(\Omega)}$ and the difference $\dps \left\| u^{\eps,G}_H -  u^{\eps,PG}_H \right\|_{H^1(\Omega)}$. As expected from Lemma~\ref{lem:diff-MsFEM-PG-estimate}, we observe that the difference $u^{\eps,G}_H - u^{\eps,PG}_H$ is extremely small (here by a factor at least of 300) in comparison to the error $u^{\eps,G}_H - u_{\rm ref}$. The intrusive and the non-intrusive MsFEMs thus share the same accuracy.

\begin{table}[ht!]
  \begin{center}
    \caption{Errors between $u^{\eps,G}_H$, $u^{\eps,PG}_H$ and $u_{\rm ref}$ for several values of $H$. \label{tab:LinGal-vs-LinPG}}
    \pgfplotstabletypeset[
  sci zerofill, % zero filling for numbers in scientific format
  sci sep align, % alignment of data at the scientific notation multiplication
  % Customize scientific notation format
  sci generic={mantissa sep={\, \times \,}, exponent={10^{##1}} },
  columns={H,LinGal,Lin}, % load columns to tabulate from the macro specified below 
  %% Specify layout for each column
  columns/H/.style={
     column name=$H/\varepsilon$,  % compute 1/H from data to be used in the table
     %
     %preproc/expr={##1*1.4142/0.02094},
     % The diameter of the mesh elements used is sqrt(2) * value in the column H
     % Nevertheless, it seems more reasonable to compare epsilon with H
     % (there is in particular a miracle if H=epsilon), hence the instruction
     preproc/expr={##1/0.02094},
     % epsilon = pi/150
     %
     %% Add some spacing at the beginning of the column for better alignment
     %% by adjusting the parameter of \hspace below
     postproc cell content/.code={% 
        \ifnum\pgfplotstablepartno=0% 
            \pgfkeys{/pgfplots/table/@cell content/.add={\hspace{0.25cm}}{}} 
        \fi%
     },
     precision=2,
    },
  columns/LinGal/.style={
     %% Add some spacing at the beginning of the column for better alignment
     %% by adjusting the parameter of \hspace below
     postproc cell content/.code={% 
        \ifnum\pgfplotstablepartno=0% 
            \pgfkeys{/pgfplots/table/@cell content/.add={\hspace{0.5cm}}{}} 
        \fi%
     },
     column name=$\left\| u^{\eps,G}_H - u_{\rm ref} \right\|_{H^1(\Omega)}$,
     sci, precision=2
    },
  columns/Lin/.style={
     %% Add some spacing at the beginning of the column for better alignment
     %% by adjusting the parameter of \hspace below
     postproc cell content/.code={% 
        \ifnum\pgfplotstablepartno=0% 
            \pgfkeys{/pgfplots/table/@cell content/.add={\hspace{0.55cm}}{}} 
        \fi%
     },
     column name=$\left\| u^{\eps,G}_H - u^{\eps,PG}_H \right\|_{H^1(\Omega)}$,
     sci, precision=2
    },
  %% Specify global layout for table
  %every even row/.style={
  %  before row={\rowcolor[gray]{0.9}}},
  every head row/.style={
      before row=\toprule,
      after row=\midrule},
  every last row/.style={
      after row=\bottomrule}
]{\GalvsPG}

  \end{center}
\end{table}

\section{Concluding remarks on possible extensions} \label{sec:conc}

Despite the fact that it significantly improves upon the classical FEM approach, the MsFEM approach using the basis functions $\phiEps{i}$ defined by~\eqref{eq:MsFEM-basis} suffers from a well-known shortcoming, due to the fact that \emph{affine} boundary conditions are imposed for $\phiEps{i}$. The method cannot yield an accurate approximation of the reference solution $u^\eps$ near the edges of the coarse mesh elements, since the exact solution oscillates along the edges while the numerical approximation does not. To overcome this drawback, several alternative definitions of the multiscale basis functions have been proposed, leading to different MsFEM variants, including the oversampling variant, introduced in~\cite{hou_multiscale_1997} and nowadays considered to be a reference MsFEM variant. We will investigate in~\cite{non_intrusif_long,biezemans_difficult_nodate} how the non-intrusive implementation procedure presented here can be extended to that case. We note that oversampling multiscale basis functions are again defined in terms of $\phiPone{i}$ in a {\em linear} manner, which should allow to obtain a formula analogous to~\eqref{eq:fondamental}. The stiffness matrices of the Galerkin and the Petrov-Galerkin approximations are in general different (in contrast to the case studied here, see Lemma~\ref{lem:MsFEMP1-macro-diff}), but formal arguments show that the two of them can be expressed as the stiffness matrix obtained using a standard $\Pone$ approximation of the single-scale problem~\eqref{eq:diffusion-pde-macro}, for an appropriately defined piecewise constant effective matrix $\overline{A}$. We refer to~\cite{non_intrusif_long,biezemans_difficult_nodate} for definite conclusions.

\smallskip

To outline the versality of the above procedure leading to a non-intrusive implementation, we eventually make a few remarks on the advection-diffusion problem
\begin{equation*}
  -\dive (A^\eps \nabla u^\eps) + b \cdot \nabla u^\eps = f \ \ \text{in $\Omega$}, \qquad u^\eps = 0 \ \ \text{on $\partial \Omega$}.
  %\label{eq:adv-diffusion-pde}
\end{equation*}
The MsFEM basis functions may be defined by~\eqref{eq:MsFEM-basis} or by a similar equation including the advection term (see e.g.~\cite{le_bris_numerical_2017,biezemans_difficult_nodate}). Oversampling may also be used. In all cases, for the same reasons as above, an identity of the type~\eqref{eq:fondamental} should hold, which allows to express the stiffness matrix of the approach as the stiffness matrix obtained using a standard $\Pone$ finite element approximation of a single-scale problem containing diffusion and advection terms, with appropriate definitions of a diffusion matrix $\overline{A}$ and an advection field $\overline{b}$ that are both piecewise constant. We again refer to~\cite{non_intrusif_long,biezemans_difficult_nodate} for details.

\section*{Acknowledgments}

The first author acknowledges the support of DIM Math INNOV. The work of the second and third authors is partially supported by ONR under grant N00014-20-1-2691 and by EOARD under grant FA8655-20-1-7043. These two authors acknowledge the continuous support from these two agencies. The fourth author thanks Inria for the financial support enabling his two-year partial leave (2020-2022) that has significantly facilitated the collaboration on this project.

\bibliographystyle{plain}
\bibliography{biblio_non_intrusif_BLLL}

\end{document}